\documentclass[preprint]{elsarticle}
\usepackage{algorithm,algorithmic}
\usepackage{amsmath}
\usepackage{amsthm}
\usepackage{amssymb}
\usepackage{color, colortbl}
\usepackage{pgfplots}
\pgfplotsset{compat=1.14}
\usepackage{tikz}

\usepackage{lineno,hyperref}
\modulolinenumbers[5]

\biboptions{sort&compress}

\newlength\myindent
\newcommand\bindent{%
  \begingroup
  \setlength{\itemindent}{\myindent}
  \addtolength{\algorithmicindent}{\myindent}
}
\newcommand\eindent{\endgroup}

\newcommand\abs[1]{\left|#1\right|}
\newcommand\norm[1]{\left\Vert#1\right\Vert}
\newcommand\diag[1]{\operatorname{diag}\left(#1\right)}
\newcommand\re[1]{\operatorname{Re}\left(#1\right)}
\newcommand\sv[2]{\operatorname{sv}_{#1}(#2)}
\newcommand\hd[2]{\operatorname{hd}(#1,#2)}
\newcommand\func[1]{\operatorname{function}~[#1]}
\DeclareMathOperator{\specr}{specR}
\DeclareMathOperator{\edger}{edgeR}

\usepackage{mathtools}
\DeclarePairedDelimiter\ceil{\lceil}{\rceil}

\newtheorem{theorem}{Theorem}[section]
\newtheorem{lemma}[theorem]{Lemma}
\newtheorem{corollary}[theorem]{Corollary}
\newtheorem{proposition}[theorem]{Proposition}

\theoremstyle{definition}
\newtheorem{example}[theorem]{Example}

\begin{document}

\begin{frontmatter}
\title{On the Graph Laplacian and the Rankability of Data}
\author{Thomas R. Cameron\fnref{dfootnote}}
\author{Amy N. Langville\fnref{cfootnote}}
\author{Heather C. Smith\fnref{dfootnote}}
\fntext[dfootnote]{Department of Mathematics and Computer Science, Davidson College, Davidson, NC (thcameron@davidson.edu, hcsmith@davidson.edu)}
\fntext[cfootnote]{Department of Mathematics, College of Charleston, Charleston, SC \,(langvillea@cofc.edu)}

\begin{abstract}
Recently, Anderson et al. (2019) proposed the concept of rankability, which refers to a dataset's inherent ability to produce a meaningful ranking of its items. 
In the same paper, they proposed a rankability measure that is based on a integer program for computing the minimum number of edge changes made to a directed graph in order to obtain a complete dominance graph, i.e., an acyclic tournament graph.
In this article, we prove a spectral-degree characterization of complete dominance graphs and apply this characterization to produce a new measure of rankability that is cost-effective and more widely applicable.
We support the details of our algorithm with several results regarding the conditioning of the Laplacian spectrum of complete dominance graphs and the Hausdorff distance between their Laplacian spectrum and that of an arbitrary directed graph with weights between zero and one.
Finally, we analyze the rankability of datasets from the world of chess and college football.
\end{abstract}

\begin{keyword}
 \texttt{directed graphs, graph Laplacian, eigenvalues, ranking, rankability, perturbation theory}
 \MSC[2010]  90C35 \sep 05C20 \sep 05C22 \sep 05C50 \sep 62F07 \sep 47A55
\end{keyword}

\end{frontmatter}

\section{Introduction}\label{sec:intro}
The ranking of data has a long and interesting history which intersects the seminal works on search engines~\cite{Brin1998}, college sports rankings~\cite{Colley2002,Massey1997}, and fair voting systems~\cite{Arrow1970}.
In addition, modern applications of ranking include movie databases, recommendation systems, social networks, and college rankings.
For a detailed account of the history of ranking and ranking methods, see~\cite{Langville2012}.

More recently, the concept of rankability was proposed by Anderson et al.~\cite{Anderson2019},
where they define a rankability measure that is based on how far a given directed graph (digraph) is from a \emph{complete dominance graph}, i.e., an acyclic tournament graph.
A complete dominance graph represents an ideal situation where all possible comparisons are explored and there is a clear and unique ranking, e.g., the complete dominance graph associated with the ranking $(1,2,3,4,5,6)$ is shown in Figure~\ref{fig:dom}.

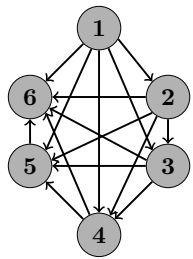
\begin{figure}[ht]
\centering
\resizebox{0.20\textwidth}{!}{
\begin{tikzpicture}
	\node[circle,draw=black,fill=black!30] (1) at (0,2) {\textbf{1}};
	\node[circle,draw=black,fill=black!30] (2) at (1,1) {\textbf{2}};
	\node[circle,draw=black,fill=black!30] (3) at (1,0) {\textbf{3}};
	\node[circle,draw=black,fill=black!30] (4) at (0,-1) {\textbf{4}};
	\node[circle,draw=black,fill=black!30] (5) at (-1,0) {\textbf{5}};
	\node[circle,draw=black,fill=black!30] (6) at (-1,1) {\textbf{6}};
	
	\draw[black,->,thick](1) to [out=330,in=135,looseness=0](2);
	\draw[black,->,thick](1) to [out=295,in=135,looseness=0](3);
	\draw[black,->,thick](1) to [out=270,in=90,looseness=0](4);
	\draw[black,->,thick](1) to [out=250,in=45,looseness=0](5);
	\draw[black,->,thick](1) to [out=225,in=45,looseness=0](6);
	\draw[black,->,thick](2) to [out=270,in=90,looseness=0](3);
	\draw[black,->,thick](2) to [out=240,in=60,looseness=0](4);
	\draw[black,->,thick](2) to [out=225,in=15,looseness=0](5);
	\draw[black,->,thick](2) to [out=180,in=0,looseness=0](6);
	\draw[black,->,thick](3) to [out=225,in=45,looseness=0](4);
	\draw[black,->,thick](3) to [out=180,in=0,looseness=0](5);
	\draw[black,->,thick](3) to [out=165,in=330,looseness=0](6);
	\draw[black,->,thick](4) to [out=135,in=315,looseness=0](5);
	\draw[black,->,thick](4) to [out=115,in=315,looseness=0](6);
	\draw[black,->,thick](5) to [out=90,in=270,looseness=0](6);
\end{tikzpicture}%
}
\caption{Complete dominance graph associated with ranking $(1,2,3,4,5,6)$.}
\label{fig:dom}
\end{figure}

In theory, the rankability measure proposed in~\cite{Anderson2019} is simple.
Given data that can be modeled as a digraph, with binary weights, let $k$ denote the minimum number of edge changes (additions or deletions) needed to obtain a complete dominance graph. 
Next, allowing for $k$ edge changes, let $p$ represent the number of complete dominance graphs that can be obtained.
Then, the rankability measure of the associated dataset is defined by
\begin{equation}\label{eq:simod-rank}
\edger = 1 - \frac{kp}{k_{max}p_{max}},
\end{equation}
where $k_{max}=(n^{2}-n)/2$ and $p_{max}=n!$.

In practice, the rankability measure from~\cite{Anderson2019}, which we denote by $\edger$, is very expensive to compute.
In fact, $k$ is computed by an integer program and $p$ is computed by an eliminative algorithm that has, in the worst case, factorial complexity.
In addition, measuring rankability based on the minimum number of edge changes necessary to obtain a complete dominance graph is limited to digraphs with binary weights. 

In this article, we propose a new cost-effective and more widely applicable measure of rankability that is based on a spectral-degree characterization of complete dominance graphs.
In particular, given data that can be modeled as a digraph with weights between zero and one, we measure the variation in the Laplacian spectrum and the out-degrees of the vertices from their known values for a complete dominance graph.
For digraphs, there are several definitions of the graph Laplacian, e.g., see~\cite{Bauer2012,Chung2005,Wu2005-1}; we adhere to the definition in~\cite{Wu2005-1}.

The spectral-degree characterization of complete dominance graphs relies on a spectral characterization of acyclic digraphs, which is an important result on its own.
In addition, we support our rankability measure algorithm with several interesting results regarding the Laplacian spectrum of complete dominance graphs.
First, we show that the Laplacian spectrum of a complete dominance graph is well-conditioned; therefore, for small perturbations the variance in the spectrum is guaranteed to be small.
Then, we show that a single edge change of a complete dominance graph changes one eigenvalue of the graph Laplacian by the weight of that edge.
Finally, we prove a sharp upper bound on the Hausdorff distance between the Laplacian spectrum of a complete dominance graph and any other digraph with weights between zero and one.

Before proceeding, we note the similarities and differences of our work and several prior related works.
First, the minimum number of edge changes in~\eqref{eq:simod-rank} is analogous to the minimum number of adjacent voter preference switches in the Dodgson method~\cite{Black1998,Hemaspaandra1997,Ratliff2010}.
However, this is not an exact similarity since there is no direct relation between edges and adjacent voter preferences.
Second, our characterization of acyclic digraphs is similar to Bauer's characterization in~\cite{Bauer2012}, but Bauer's characterization is for the normalized Laplace operator and is not related to the out-degrees of the vertices, as ours is.
Finally, some may find our work to be reminiscent of Landau's work on dominance relations and tournament graphs~\cite{Brauer1968,Harary1966,Landau1953}, but our focus on acyclic tournament graphs requires a stronger statement of both the out-degrees of the vertices and the eigenvalues of the graph Laplacian.

\section{The Graph Laplacian}\label{sec:laplacian}
Let $\mathbb{G}$ denote the set of finite simple digraphs with non-negative weights.
For each $\Gamma\in\mathbb{G}$, we have $\Gamma=(V,E,w)$, where $V=\left\{1,2,\ldots,n\right\}$ is the vertex set, $E\subseteq V\times V$ is the edge set, and $w\colon V\times V\rightarrow\mathbb{R}_{\geq 0}$ is the associated weight function.
If $(i,j)\in E$, then there is an edge from $i$ to $j$; the weight of the edge $(i,j)$ is given by $w_{ij}$.
We use the convention that $w_{ij}=0$ if and only if $(i,j)\notin E$.

A \emph{tournament} is a digraph in $\mathbb{G}$ such that for each $i,j\in V$ either $(i,j)\in E$ or $(j,i)\in E$, but not both. On the other hand, a digraph is \emph{acyclic} if it does not contain a set of edges $\{(x_i,y_i):i\in \{1,2,\ldots, k\}\}$ where $y_i=x_{i+1}$ for each $i<k$ and $y_{k}=x_1$. In this section, we prove the spectral degree characterization of acyclic digraphs and the spectral-degree characterization of \emph{complete dominance graphs}, which are acyclic tournaments where the weight of each edge is 1.

Given $\Gamma\in\mathbb{G}$, we define the \emph{out-degree} of the vertex $i\in V$ by 
\[
d^{+}(i)=\sum_{j\in V}w_{ij}.
\]
Furthermore, we define the \emph{out-degree matrix} by
\[
D=\diag{d^{+}(1),\ldots,d^{+}(n)},
\]
i.e., the diagonal matrix with entries $d^{+}(1),\ldots,d^{+}(n)$.
In addition, we define the \emph{weighted adjacency matrix} by 
\[
A=[w_{ij}]_{i,j=1}^{n},
\]
i.e., the $n\times n$ matrix whose $(i,j)$ entry corresponds to the weight $w_{ij}$. 
Then, we define the \emph{graph Laplacian} by
\[
L=D-A.
\]

Let $\Gamma=(V,E,w)\in\mathbb{G}$ and let $V'\subseteq V$.
Then, we define the induced subgraph $\Gamma'$ of $\Gamma$ as the digraph with vertex set $V'$, edge set $E'=E\cap(V'\times V')$, and weight function $w'=w\restriction_{V'\times V'}$.
We say that $\Gamma'$ is \emph{isolated} if $w_{ij}=0$ for all $i\in V'$ and $j\notin V'$.
A subgraph of $\Gamma$ is \emph{strongly connected} if for each pair of vertices $i$ and $j$, either $i=j$ or there is a directed path from $i$ to $j$ and a directed path from $j$ to $i$. A \emph{strongly connected component} of $\Gamma$ is a maximal strongly connected subgraph.

Note that a digraph $\Gamma$ can be uniquely decomposed into strongly connected components.
Also, the graph Laplacian, possibly after reordering of the vertices, can be written in Frobenius normal form~\cite{Brualdi1991}:
\begin{equation}\label{eq:fro-form}
L = \begin{bmatrix} L_{1} & L_{12} & \cdots & L_{1r} \\
				& L_{2} & \cdots & L_{2r} \\
				& & \ddots & \vdots \\
				& & & L_{r} \end{bmatrix},
\end{equation}
where the blocks $L_{k}$ are irreducible matrices that correspond to the strongly connected components $\Gamma_{k}$ of $\Gamma$. 
Let $V_{k}$ denote the vertex set of the strongly connected component $\Gamma_{k}$.
Then, the submatrices $L_{kl}$, $1\leq k< l\leq r$, contain elements of the form $-w_{ij}$ for all $i\in V_{k}$ and $j\in V_{l}$.

\subsection{Basic Spectral Properties}\label{subsec:basic-prop}
Given $\Gamma\in\mathbb{G}$, we denote the spectrum of the graph Laplacian $L$ by $\sigma(L)$.
First, we note several basic properties of the graph Laplacian and its spectrum.
Due to the elementary nature of these properties, we omit their proof. 

\begin{proposition}\label{prop:basic-spectrum}
Let $\Gamma\in\mathbb{G}$ and let $L$ be the graph Laplacian of $\Gamma$.
Then, the following properties hold.
\begin{enumerate}[(i)]
\item	Zero is an eigenvalue of $L$.
\item	The spectrum of $L$ is symmetric with respect to the real axis.
\item	Let $\lambda_{1},\ldots,\lambda_{n}$ denote the eigenvalues of $L$.
	Then, 
	\[
	\sum_{i=1}^{n}\lambda_{i} = \sum_{i=1}^{n}\re{\lambda_{i}} = \sum_{i=1}^{n}d^{+}(i).
	\]
\item	Let $L$ be in Frobenius normal form~\eqref{eq:fro-form}.
	Then,
	\[
	\sigma(L) = \bigcup_{i=1}^{r}\sigma(L_{i}). 
	\]
\item	The graph Laplacian $L$ is an M-matrix.
	Furthermore, if $L$ is in Frobenius normal form~\eqref{eq:fro-form}, then each $L_{k}$, $1\leq k\leq r$, is an irreducible M-matrix. 
\end{enumerate}
\end{proposition}

Next, we consider the relationship between the zero eigenvalues of the graph Laplacian and the isolated strongly connected components of the digraph.
Note that we make use of the following famous result from Taussky~\cite[Theorem II]{Taussky1949} in the proof of Lemma~\ref{lem:sc-zero-eig}.

\begin{theorem}\label{thm:taussky}
The complex matrix $A=[a_{ij}]_{i,j=1}^{n}$ is non-singular provided that $A$ is irreducible and
\[
\abs{a_{ii}}\geq\sum_{j\neq i}\abs{a_{ij}},
\]
is satisfied for all $i=1,\ldots,n$, with equality in at most $(n-1)$ cases. 
\end{theorem}

\begin{lemma}\label{lem:sc-zero-eig}
Let $\Gamma\in\mathbb{G}$ and let the graph Laplacian $L=[l_{ij}]_{i,j=1}^{n}$ of $\Gamma$ be in Frobenius normal form~\eqref{eq:fro-form}.
If $\Gamma_{k}$ is not isolated, then zero is not an eigenvalue of $L_{k}$.
Furthermore, if $\Gamma_{k}$ is isolated, then zero is a simple eigenvalue of $L_{k}$.
\end{lemma}
\begin{proof}
Suppose that $\Gamma_{k}$ is not isolated.
Then, there exists a vertex $i\in V_{k}$ such that $w_{ij}\neq 0$ for some $j\notin V_{k}$.
Therefore, since the weights are non-negative, we have
\[
\abs{l_{ii}}=d^{+}(i) > \sum_{j\in V_{k}}w_{ij}=\sum_{j\in V_{k}}\abs{l_{ij}}.
\]
For all other $i\in V_{k}$, we have
\[
\abs{l_{ii}}=d^{+}(i)\geq\sum_{j\in V_{k}}w_{ij}=\sum_{j\in V_{k}}\abs{l_{ij}}.
\]
Hence, by Theorem~\ref{thm:taussky}, it follows that $L_{k}$ is non-singular.

Now, suppose that $\Gamma_{k}$ is isolated.
Then, the row sums of $L_{k}$ are zero and it follows that $L_{k}$ is a singular irreducible M-matrix. 
Therefore, by~\cite[Theorem 6.4.16]{Berman1994}, it follows that zero is a simple eigenvalue of $L_{k}$.
\end{proof}

We conclude this section by noting the following theorem which is a direct consequence of Lemma~\ref{lem:sc-zero-eig} and the Frobenius normal form. 

\begin{theorem}\label{thm:isolated}
Let $\Gamma\in\mathbb{G}$ and let $L$ be the graph Laplacian of $\Gamma$.
Then, for $k\in\mathbb{N}$, the following statements are equivalent:
\begin{enumerate}[(i)]
\item	The algebraic multiplicity of the zero eigenvalue of $L$ is equal to $k$.
\item	There are exactly $k$ isolated strongly connected components of $\Gamma$.
\item	The geometric multiplicity of the zero eigenvalue of $L$ is equal to $k$. 
\end{enumerate}
\end{theorem}

\subsection{Acyclic Digraphs}\label{subsec:acyclic}
Now, we are ready to prove the spectral characterization of acyclic digraphs and then the spectral-degree characterization of complete dominance graphs. 
First, note the following proposition, which follows readily from the Frobenius normal form. 

\begin{proposition}\label{prop:acyclic-fro}
The following statements are equivalent:
\begin{enumerate}[(i)]
\item	The digraph $\Gamma\in\mathbb{G}$ is acyclic.
\item	Every strongly connected component of $\Gamma$ consists of exactly one vertex. 
\item	The graph Laplacian $L$ of $\Gamma$ in Frobenius normal form is upper triangular.  
\end{enumerate}
\end{proposition}

\begin{theorem}\label{thm:acyclic-spec}
Let $\Gamma\in\mathbb{G}$ and let $L$ be the graph Laplacian of $\Gamma$.
Then,
\begin{equation}\label{eq:acyclic-spec}
\sigma(L)=\left\{d^{+}(1),\ldots,d^{+}(n)\right\}
\end{equation}
if and only if $\Gamma$ is acyclic.
\end{theorem}
\begin{proof}
Suppose that $\Gamma$ is acyclic.
Then, it follows from Proposition~\ref{prop:acyclic-fro} that the Frobenius normal form of $\Gamma$ is upper triangular. 
Since the eigenvalues of an upper triangular matrix are its main-diagonal entries, it follows that the spectrum of $L$ satisfies~\eqref{eq:acyclic-spec}.

Conversely, for the sake of contradiction, suppose that the spectrum of $L$ satisfies~\eqref{eq:acyclic-spec} and $\Gamma$ is not acyclic.
Then, there is a strongly connected component of $\Gamma$ made up of at least two vertices.
Without loss of generality, let $\Gamma_{1},\ldots,\Gamma_{p}$ be the strongly connected components of $\Gamma$ that consist of at least two vertices.
Since the remaining strongly connected components of $\Gamma$ have exactly one vertex, it follows that the eigenvalues of $L_{1}\oplus\cdots\oplus L_{p}$ correspond to the out-degrees of the vertices of $\Gamma_{1}\cup\cdots\cup\Gamma_{p}$.

Suppose that $\Gamma_{k}$ is isolated for some $1\leq k\leq p$. 
Then, by Lemma~\ref{lem:sc-zero-eig}, zero is an eigenvalue of $L_{k}$ and, therefore, of $L_{1}\oplus\cdots\oplus L_{p}$.
However, since the eigenvalues of $L_{1}\oplus\cdots\oplus L_{p}$ must correspond to the out-degrees of the vertices of $\Gamma_{1}\cup\cdots\cup\Gamma_{p}$, which are all non-zero, this is a contradiction.

Now, suppose that $\Gamma_{k}$ is not isolated for all $k=1,\ldots,p$.
By assumption, the smallest eigenvalue $\lambda_{0}$ of $L_{1}\oplus\cdots\oplus L_{p}$ is equal to the smallest out-degree of the vertices of $\Gamma_{1}\cup\cdots\cup\Gamma_{p}$.
Let $B=L_{i}$, $1\leq i\leq p$, such that $\lambda_{0}$ is a main-diagonal entry of $B$.
Also, without loss of generality, we assume that the main-diagonal entries of $B$ are in non-decreasing order. 

Since $B$ is an irreducible non-singular M-matrix, it follows from~\cite[Theorem 6.2.3 (N38)]{Berman1994} that $B^{-1}$ exists, is positive, and is irreducible. 
Furthermore, by the Perron-Frobenius theorem~\cite[Theorem 2.1.4]{Berman1994}, it follows that $\rho(B^{-1})$ is a simple eigenvalue of $B^{-1}$ with a positive eigenvector $\textbf{x}$.
In addition, we normalize the positive eigenvector $\textbf{x}$ so that $\textbf{x}_{1}=1$. 
Therefore, the eigenvector equation
\[
B\textbf{x} = \frac{1}{\rho(B^{-1})}\textbf{x}
\]
implies that 
\begin{equation}\label{eq:eigvec}
b_{11}-b_{12}\textbf{x}_{2} -\cdots-b_{1n_{i}}\textbf{x}_{n_{i}} = \frac{1}{\rho(B^{-1})},
\end{equation}
where $n_{i}$ is the size of $L_{i}$ and the entries $b_{12},\ldots,b_{1n_{i}}$ are non-negative. 

Now, we have two cases to consider: $\lambda_{0}\in\sigma(B)$ and $\lambda_{0}\notin\sigma(B)$.
If $\lambda_{0}\in\sigma(B)$, then $b_{11}=1/\rho(B^{-1})$ and it follows from~\eqref{eq:eigvec} that $b_{12},\ldots,b_{1n_{i}}$ are all zero, which contradicts $\Gamma_{i}$ being a strongly connected component made up of at least two vertices. 
If $\lambda_{0}\notin\sigma(B)$, then $b_{11}<1/\rho(B^{-1})$, which directly contradicts~\eqref{eq:eigvec}.
\end{proof}

The spectral-degree characterization of complete dominance graphs is stated in the corollary below.

\begin{corollary}\label{cor:dom-eigval}
Let $\Gamma\in\mathbb{G}$ have weights $0\leq w_{ij}\leq 1$ and let $L$ be the graph Laplacian of $\Gamma$. 
Then, $\Gamma$ is a complete dominance graph if and only if
\[
\sigma(L)=\left\{d^{+}(1),d^{+}(2),\ldots,d^{+}(n)\right\}=\left\{n-1,n-2,\ldots,0\right\}.
\]
\end{corollary}
\begin{proof}
Suppose that $\Gamma$ is a complete dominance graph.
Then, since $\Gamma$ is acyclic, it follows from Theorem~\ref{thm:acyclic-spec} that $\sigma(L)=\left\{d^{+}(1),\ldots,d^{+}(n)\right\}$.
In addition, as a tournament graph with binary weights and no cycles, there exists a re-ordering of the vertices such that $d^{+}(i)=n-i$, for $i=1,\ldots,n$.

Conversely, suppose that
\[
\sigma(L)=\left\{d^{+}(1),d^{+}(2),\ldots,d^{+}(n)\right\}=\left\{n-1,n-2,\ldots,0\right\}.
\]
Then, by Theorem~\ref{thm:acyclic-spec}, it follows that $\Gamma$ is acyclic.
Furthermore, there is a re-ordering of the vertices such that $d^{+}(i)=n-i$, for $i=1,\ldots,n$.
Since the weights satisfy $0\leq w_{ij}\leq 1$, it follows that $(1,j)\in E$ with weight $w_{1j}=1$, for $j=2,\ldots,n$.
Similarly, since $\Gamma$ is acyclic, it follows that $(2,1)\notin E$ and $(2,j)\in E$ with weight $w_{2j}=1$, for $j=3,\ldots,n$.
Continuing in this fashion, we conclude that $\Gamma$ is a complete dominance graph.
\end{proof}

We emphasize that the eigenvalues alone are not enough to characterize a complete dominance graph.
As the following example illustrates, there exist digraphs that are not isomorphic to a complete dominance graph, yet share the same Laplacian spectrum. 

\begin{example}\label{ex:dom-eigval}
Let $\Gamma$ denote the digraph with binary weights in Figure~\ref{fig:dom-eigval}.
Let $L$ denote the graph Laplacian of $\Gamma$ and note that $\sigma(L)=\left\{2,1,0\right\}$,
which is equal to the Laplacian spectrum of any complete dominance graph with three vertices. 
However, $\Gamma$ is not a complete dominance graph since it has a cycle between the vertices $1$ and $3$.
This example does not contradict Corollary~\ref{cor:dom-eigval} since the eigenvalues of $L$ are not equal to the out-degrees of the vertices of $\Gamma$.

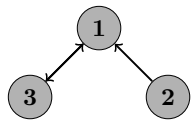
\begin{figure}[ht]
\centering
\resizebox{0.20\textwidth}{!}{
\begin{tikzpicture}
	\node[circle,draw=black,fill=black!30] (1) at (0,1) {\textbf{1}};
	\node[circle,draw=black,fill=black!30] (2) at (1,0) {\textbf{2}};
	\node[circle,draw=black,fill=black!30] (3) at (-1,0) {\textbf{3}};
	
	\draw[black,->,thick](1) to [in=45,out=225,looseness=0](3);
	\draw[black,->,thick](2) to [in=315,out=135,looseness=0](1);
	\draw[black,->,thick](3) to [in=225,out=45,looseness=0](1);
\end{tikzpicture}%
}
\caption{Digraph with binary weights on 3 vertices.}
\label{fig:dom-eigval}
\end{figure}
\hfill$\qedsymbol$
\end{example}

\section{A Rankability Measure}\label{sec:rankability}
Armed with the spectral-degree characterization (Corollary~\ref{cor:dom-eigval}) of complete dominance graphs, we propose a new rankability measure. 
In particular, given data that can be modeled as a digraph with weights $0\leq w_{ij}\leq 1$, we measure the variation in the Laplacian spectrum and the out-degrees of the vertices from their known values for a complete dominance graph. 

We measure this variation using the Hausdorff distance.
Let $A$ and $\tilde{A}$ be complex matrices with eigenvalues $\lambda_{1},\ldots,\lambda_{n}$ and $\tilde{\lambda}_{1},\ldots,\tilde{\lambda}_{n}$, respectively.
The \emph{Hausdorff distance} between the eigenvalues of $A$ and $\tilde{A}$ is defined as follows~\cite{Stewart1990}:
\[
\hd{A}{\tilde{A}} = \max\left\{\sv{A}{\tilde{A}},\sv{\tilde{A}}{A}\right\},
\]
where the so called \emph{spectral variation} is defined by
\[
\sv{A}{\tilde{A}} = \max_{i}\min_{j}\abs{\tilde{\lambda}_{i}-\lambda_{j}},~\quad~1\leq i,j\leq n.
\]

We note that, in general, the Hausdorff distance is a metric on the set of subsets of a metric space~\cite{Munkres1999}.
However, for our purposes, we only need the Hausdorff distance to measure the variation in finite multi-sets of complex numbers. 
In the algorithm below, we show how the Hausdorff distance is used to measure rankability.
We denote our rankability measure by $\specr$ since it is based on the spectral-degree characterization of complete dominance graphs.

\begin{algorithm}[ht]
\caption{Spectral-Degree Rankability of Graph Data $\Gamma$.}
\label{alg:specr}
\begin{algorithmic}
\STATE{$\func{r} = \specr\left(\Gamma\right):$}
\bindent
\STATE{$n\gets$ the number of vertices in $\Gamma$}
\STATE{$D\gets$ the out-degree matrix of $\Gamma$}
\STATE{$L\gets$ graph Laplacian of $\Gamma$}
\STATE{$S=\diag{n-1,n-2,\ldots,0}$}
\STATE{$r=1 - \frac{\hd{D}{S}+\hd{L}{S}}{2(n-1)}$}
\RETURN
\eindent
\end{algorithmic}
\end{algorithm} 

In what follows, we prove several properties of complete dominance graphs that justify the functionality of $\specr$.
In particular, we show that the Laplacian spectrum of a complete dominance graph is extremely well-conditioned; thus, for small perturbations, the corresponding Hausdorff distance is guaranteed to be small.
Then, we show that a single edge change of a complete dominance graph results in a Hausdorff distance equal to the weight of that edge.
We also show that the Hausdorff distance between the Laplacian spectrum of a complete dominance graph and any other digraph in $\mathbb{G}$ with weights $0\leq w_{ij}\leq 1$ is bounded above by $(n-1)$.
Note that this bound also holds for the out-degrees of the vertices; hence, the division by $2(n-1)$ in Algorithm~\ref{alg:specr} is a normalization factor.
The return value of $\specr$ varies between zero and one, which represents a transition between  ill-rankable and well-rankable data, respectively.
Finally, we compute the rankability of several structured digraph examples from~\cite{Anderson2019}.

\subsection{Condition Number}\label{subsec:cond-number}
By Corollary~\ref{cor:dom-eigval}, the Laplacian spectrum of the complete dominance graph is simple, i.e., the eigenvalues all have an algebraic multiplicity equal to one.
Therefore, we can make use of the following perturbation theorem~\cite[Theorem 7.1.12]{Watkins2010}.

\begin{theorem}\label{thm:watkins}
Let $A$ be an $n\times n$ complex matrix with simple eigenvalues.
Let $\lambda$ be an eigenvalue of $A$ with right and left eigenvectors $\textbf{\emph{v}}$ and $\textbf{\emph{w}}$, respectively, normalized so that $\norm{\textbf{\emph{v}}}_{2}=\norm{\textbf{\emph{w}}}_{2}=1$.
Also, let  $E$ be a small perturbation satisfying $\norm{E}_{2}=\epsilon$.
Then, there exists an eigenvalue $\lambda+\Delta\lambda$ of $A+E$ such that the condition number $\kappa(\lambda)=1/\abs{\textbf{\emph{w}}^{*}\textbf{\emph{v}}}$ satisfies
\[
\abs{\Delta\lambda}\leq\kappa(\lambda)\epsilon + O(\epsilon^{2}).
\]
\end{theorem}

In order to determine the condition number of the eigenvalues of the graph Laplacian of a complete dominance graph, we make use of the following variant of Corollary~\ref{cor:dom-eigval}.

\begin{theorem}\label{thm:dom-eigvec}
Let $\Gamma\in\mathbb{G}$ have weights $0\leq w_{ij}\leq 1$ and let $L$ be the graph Laplacian of $\Gamma$.
Then, $\Gamma$ is a complete dominance graph if and only if there exists a permutation matrix $P$ such that, for $i=1,\ldots,n$, $P\textbf{\emph{v}}_{i}$ is an eigenvector of $L$ for the eigenvalue $(n-i)$, where
\begin{equation}\label{eq:dom-eigvec}
\textbf{\emph{v}}_{i}=\sum_{k=1}^{i}\textbf{\emph{e}}_{k}
\end{equation}
and $\textbf{\emph{e}}_{k}$ is the $k$th standard basis vector of $\mathbb{R}^{n}$.
\end{theorem}
\begin{proof}
Suppose that $\Gamma$ is a complete dominance graph. 
Then, it follows from Corollary~\ref{cor:dom-eigval} that there exists a permutation matrix $P$ such that $P^{T}LP$ is in Frobenius normal form~\eqref{eq:fro-form} with diagonal blocks
\[
L_{1}=[n-1], L_{2}=[n-2],\ldots,L_{n}=[0]
\]
and all negative ones above the main-diagonal.
Now, for $i=1,\ldots,n$, it is readily verified that $\textbf{v}_{i}$ as defined in~\eqref{eq:dom-eigvec} is an eigenvector of $P^{T}LP$ for the eigenvalue $(n-i)$ and, therefore,
\[
LP\textbf{v}_{i}=(n-i)P\textbf{v}_{i}.
\]

Conversely, suppose that there exists a permutation matrix $P$ such that, for $i=1,\ldots,n$, $P\textbf{v}_{i}$ is an eigenvector of $L$ for the eigenvalue $(n-i)$, where $\textbf{v}_{i}$ is defined in~\eqref{eq:dom-eigvec}.
Then,
\[
P^{T}LP\textbf{e}_{1}=(n-1)\textbf{e}_{1}
\]
and it follows that, under this re-ordering of the vertices, $d^{+}(1)=(n-1)$.
Next, we have
\[
P^{T}LP(\textbf{e}_{1}+\textbf{e}_{2})=(n-2)(\textbf{e}_{1}+\textbf{e}_{2}),
\]
which implies that
\[
P^{T}LP\textbf{e}_{2}=-\textbf{e}_{1}+(n-2)\textbf{e}_{2}.
\]
Hence, under this re-ordering of the vertices, $d^{+}(2)=(n-2)$. 
Continuing in this fashion, we have a re-ordering of the vertices such that $d^{+}(i)=n-i$, for $i=1,\ldots,n$.
In addition, it follows that the eigenvalues of $L$ satisfy
\[
\sigma(L)=\left\{d^{+}(1),d^{+}(2),\ldots,d^{+}(n)\right\}=\left\{n-1,n-2,\ldots,0\right\}.
\]
Therefore, by Corollary~\ref{cor:dom-eigval}, $\Gamma$ is a complete dominance graph. 
\end{proof}

In addition, we have the following readily verified proposition. 

\begin{proposition}\label{prop:dom-left-eigvec}
Let $\Gamma\in\mathbb{G}$ be a complete dominance graph and let $L$ be the graph Laplacian of $\Gamma$.
Denote by $V$ the matrix whose $i$th column vector is $\textbf{\emph{v}}_{i}$ as defined in~\eqref{eq:dom-eigvec}. 
Then, the inverse of $V$ satisfies
\begin{equation}\label{eq:left-eigvec}
V^{-1}=\begin{bmatrix}1 & -1 & \cdots & 0 \\
		\vdots & \ddots & \ddots &  \vdots \\
		0 & \cdots & 1 & -1 \\
		0 & \cdots & 0 & 1
		\end{bmatrix}.
\end{equation}
Moreover, there exists a permutation matrix $P$ such that, for $i=1,\ldots,n$, the $i$th row of $V^{-1}$ is a left eigenvector of $P^{T}LP$ for the eigenvalue $(n-i)$.  
\end{proposition}

Finally, Corollary~\ref{cor:dom-eigval-cond} follows immediately from Theorem~\ref{thm:watkins}, Theorem~\ref{thm:dom-eigvec}, and Proposition~\ref{prop:dom-left-eigvec}. 

\begin{corollary}\label{cor:dom-eigval-cond}
Let $\Gamma\in\mathbb{G}$ be a complete dominance graph and let $L$ be the graph Laplacian of $\Gamma$.
For $i=1,\ldots,n$, denote the $i$th eigenvalue of $L$ by $\lambda_{i}=n-i$.
Then, the condition number of $\lambda_{i}$ satisfies
\[
\kappa(\lambda_{i})=\sqrt{2i},
\]
for $i=1,\ldots,n$. 
\end{corollary}

Let $\Gamma\in\mathbb{G}$ be a complete dominance graph and let $L$ be the graph Laplacian of $\Gamma$.
Define $\tilde{L}=L+E$, where $E$ is some perturbation that satisfies $\norm{E}_{2}=\epsilon$.
Then, by Theorem~\ref{thm:watkins} and Corollary~\ref{cor:dom-eigval-cond}, for any eigenvalue $\lambda$ of $L$ there is an eigenvalue $\lambda+\Delta\lambda$ of $\tilde{L}$ such that
\begin{equation}\label{eq:dom-spec-bound}
\abs{\Delta\lambda}\leq \sqrt{2n}\epsilon + O(\epsilon^{2}).
\end{equation}
Thus, the eigenvalues of $L$ are well-conditioned, i.e., for small perturbations $E$ (and reasonably sized $n$), the Hausdorff distance between the eigenvalues of $L$ and $\tilde{L}$ is guaranteed to be small.
If, in addition, we compute the eigenvalues of $\tilde{L}$ using the QR algorithm, which is well-known to be backward stable~\cite{Stewart1973}, then we can guarantee the computed Hausdorff distance $\hd{L}{\tilde{L}}$ is small provided that $\norm{E}_{2}$ is small.

\subsection{Sharp Bounds}\label{subsec:sharp-bounds}
The bound in~\eqref{eq:dom-spec-bound} depends on an arbitrary small perturbation $E$.
However, we are particularly interested in perturbations $E$ that reflect edge changes so that $\tilde{L}=L+E$ is the graph Laplacian of $\tilde{\Gamma}\in\mathbb{G}$ with weights $0\leq w_{ij}\leq 1$.
In this section, we obtain sharp bounds on the Hausdorff distance between the Laplacian spectrum of a given digraph and a complete dominance graph. 

We also show that a single edge change of a complete dominance graph results in a Hausdorff distance equal to the weight of that edge. 
This result relies on the following relation between the determinant of a matrix and its rank-one perturbation~\cite[Lemma 1.1]{Ding2007}.

\begin{lemma}\label{lem:ding}
If $A$ is an invertible $n\times n$ complex matrix and $\textbf{\emph{x}},\textbf{\emph{y}}\in\mathbb{C}^{n}$, then
\[
\det(A+\textbf{\emph{x}}\textbf{\emph{y}}^{T})=(1+\textbf{\emph{y}}^{T}A^{-1}\textbf{\emph{x}})\det(A).
\]
\end{lemma}

Moreover, we have the following extension of~\cite[Theorem 2.1]{Ding2007}.
 
\begin{theorem}\label{thm:ding-var}
Let $A$ be a $n\times n$ complex matrix with eigenvalues
\[
\sigma(A)=\left\{\lambda_{1},\lambda_{2},\ldots,\lambda_{n}\right\}.
\]
Furthermore, let $\textbf{\emph{w}}_{1},\ldots,\textbf{\emph{w}}_{k}$ be left eigenvectors for the eigenvalues $\lambda_{1},\ldots,\lambda_{k}$, respectively.
If $\textbf{\emph{x}}\in\mathbb{C}^{n}$ is perpendicular to $\textbf{\emph{w}}_{2},\ldots,\textbf{\emph{w}}_{k}$, then
\[
\sigma(A+\textbf{\emph{x}}(\textbf{\emph{w}}_{1}+\cdots+\textbf{\emph{w}}_{k})^{*})=\left\{\lambda_{1}+\textbf{\emph{w}}_{1}^{*}\textbf{\emph{x}},\lambda_{2},\ldots,\lambda_{n}\right\}.
\]
\end{theorem}
\begin{proof}
Let $\lambda$ be any complex number that is not an eigenvalue of $A$, and let $\textbf{w}=\textbf{w}_{1}+\cdots+\textbf{w}_{k}$. 
Then, by applying Lemma~\ref{lem:ding} to the equality
\[
\lambda I - (A+\textbf{x}\textbf{w}^{*})=(\lambda I-A)-\textbf{x}\textbf{w}^{*},
\]
we have 
\[
\det(\lambda I-(A+\textbf{x}\textbf{w}^{*})) = (1-\textbf{w}^{*}(\lambda I-A)^{-1}\textbf{x})\det(\lambda I - A).
\]

For $i=1,\ldots,k$, $\textbf{w}_{i}$ is a left eigenvector for the eigenvalue $\lambda_{i}$, and it follows that
\[
\textbf{w}_{i}^{*}(\lambda I - A)^{-1} = \frac{1}{\lambda-\lambda_{i}}\textbf{w}_{i}^{*}.
\]
Therefore, we have
\begin{align*}
1-\textbf{w}^{*}(\lambda I-A)^{-1}\textbf{x} &= 1 - \frac{\textbf{w}_{1}^{*}\textbf{x}}{\lambda-\lambda_{1}} - \cdots - \frac{\textbf{w}_{k}^{*}\textbf{x}}{\lambda-\lambda_{k}} \\
&= \frac{(\lambda-\lambda_{1})-\textbf{w}_{1}^{*}\textbf{x}}{\lambda-\lambda_{1}}
\end{align*}
since $\textbf{x}$ is perpendicular to $\textbf{w}_{2},\ldots,\textbf{w}_{k}$.
Hence,
\[
\det( \lambda I-(A+\textbf{x}\textbf{w}^{*})) = (\lambda - (\lambda_{1}+\textbf{w}_{1}^{*}\textbf{x}))(\lambda-\lambda_{2})\cdots(\lambda-\lambda_{n}),
\]
and the result follows.
\end{proof}

We are now ready to prove that a single edge change of a complete dominance graph results in a Hausdorff distance equal to the weight of that edge.

\begin{theorem}\label{thm:haus-pert-edge}
Let $\Gamma\in\mathbb{G}$ be a complete dominance graph and let $L$ be the graph Laplacian of $\Gamma$.
Let $\tilde{\Gamma}$ be obtained from $\Gamma$ by adding or deleting an edge with weight $w$ and let $\tilde{L}$ be the graph Laplacian of $\tilde{\Gamma}$.
Then,
\[
\hd{L}{\tilde{L}}=w.
\]
\end{theorem}
\begin{proof}
Note that if $\tilde{\Gamma}$ is acyclic, then the result follows readily from Theorem~\ref{thm:acyclic-spec}.
Thus, we only need to consider edge additions that result in a cycle. 

Without loss of generality, we assume that $L$ is in its Frobenius normal form.
The eigenvalues of $L$ satisfy
\[
\sigma(L)=\left\{n-1,n-2,\ldots,0\right\},
\]
with associated left eigenvectors $\textbf{w}_{1},\textbf{w}_{2},\ldots,\textbf{w}_{n}$ that correspond to the rows of $S^{-1}$ in~\eqref{eq:left-eigvec}.

For $i=2,\ldots,n$ and $j=1,\ldots,i-1$, the edge $(i,j)$ with weight $w_{ij}$ can be added by the rank-one perturbation
\[
w_{ij}\textbf{e}_{i}(\textbf{w}_{j}+\cdots+\textbf{w}_{i-1})^{*}
\]
Note that $\textbf{e}_{i}$ is perpendicular to $\textbf{w}_{k}$ for $k=j,\ldots,i-2$.
Therefore, Theorem~\ref{thm:ding-var} applies to the perturbed dominance graph Laplacian
\[
\tilde{L}=L+w_{ij}\textbf{e}_{i}(\textbf{w}_{j}+\cdots+\textbf{w}_{i-1})^{*}.
\]
Since $\textbf{w}_{i-1}^{*}\textbf{e}_{i}=1$, we have
\[
\sigma(\tilde{L})=\left\{n-1,\ldots,n-(i-1)+w_{ij},\ldots,0\right\},
\]
and the result follows.
\end{proof}

Note that the result in Theorem~\ref{thm:haus-pert-edge} can also be viewed in the context of changing the weight of an existing edge. 
Finally, we prove a sharp upper bound on the Hausdorff distance.
To this end, we make use of the following famous result due to Ger{\v s}gorin~\cite{Gershgorin1931}.

\begin{theorem}\label{thm:gersh}
Let $A=[a_{ij}]_{i,j=1}^{n}$ be a complex matrix and let
\[
R_{i}(A)=\sum_{j\neq i}\abs{a_{ij}},~\quad~i=1,\ldots,n.
\]
For $i=1,\ldots,n$, define the $i$th Ger{\v s}gorin disk by
\[
G_{i}(A) = \left\{z\in\mathbb{C}\colon\abs{z-a_{ii}}\leq R_{i}(A)\right\}.
\]
Then, the eigenvalues of $A$ are in the union of the Ger{\v s}gorin disks:
\[
G(A)=\bigcup_{i=1}^{n}G_{i}(A). 
\]
\end{theorem}

\begin{theorem}\label{thm:haus-upper-bound}
Let $\Gamma\in\mathbb{G}$ be a complete dominance graph and let $\tilde{\Gamma}\in\mathbb{G}$ be any other graph with weights $0\leq \tilde{w}_{ij}\leq 1$. 
Denote by $L$ and $\tilde{L}$ the graph Laplacian of $\Gamma$ and $\tilde{\Gamma}$, respectively.
Then,
\[
\hd{L}{\tilde{L}}\leq (n-1).
\]
\end{theorem}
\begin{proof}

Let $\tilde{d}^{+}(i)$ denote the out-degree of the $i$th vertex of $\tilde{L}$.
Then, the $i$th Ger{\v s}gorin disk of $\tilde{L}$ satisfies
\[
G_{i}(\tilde{L})=\left\{z\in\mathbb{C}\colon\abs{z-\tilde{d}^{+}(i)}\leq \tilde{d}^{+}(i)\right\}.
\]
Since the weights of $\tilde{\Gamma}$ satisfy $0\leq\tilde{w}_{ij}\leq 1$, it follows that $k=\ceil{\tilde{d}^{+}(i)}$ is an integer between $1$ and $(n-1)$.
Then, by Corollary~\ref{cor:dom-eigval}, $k$ is the out-degree of a vertex of $\Gamma$ and is an eigenvalue of $L$. 
Also, the disk
\[
D_{k}(k)=\left\{z\in\mathbb{C}\colon\abs{z-k}\leq k\right\}
\]
contains the $i$th Ger{\v s}gorin disk of $\tilde{L}$. 
Therefore, $\sv{L}{\tilde{L}}\leq k\leq (n-1)$. 

In addition, the eigenvalues of $\tilde{L}$ are guaranteed to be contained in the disk centered at $(n-1)$ with radius $(n-1)$ since $\tilde{d}^{+}(i)\leq(n-1)$, for $i=1,\ldots,n$. 
By Corollary~\ref{cor:dom-eigval}, $(n-1)$ is an eigenvalue of $L$ and it follows that $\sv{\tilde{L}}{L}\leq(n-1)$.
Since $\hd{L}{\tilde{L}}$ is the maximum of $\sv{L}{\tilde{L}}$ and $\sv{\tilde{L}}{L}$, the result follows. 
\end{proof}

Note that the upper bound in Theorem~\ref{thm:haus-upper-bound} is attained for the empty graph.

\subsection{Examples}\label{subsec:examples}
In this section, we consider several structured digraph examples from~\cite{Anderson2019}.
These examples are used to illustrate whether or not a rankability measure matches our intuition of well-rankable and ill-rankable graph data. 
For each digraph, we compare $\specr$ with the rankability measure $\edger$ as defined in~\eqref{eq:simod-rank}.
The results are displayed in Figure~\ref{fig:simod-examples}, and the digraphs are ordered from most rankable to least rankable as determined by the measure $\specr$.

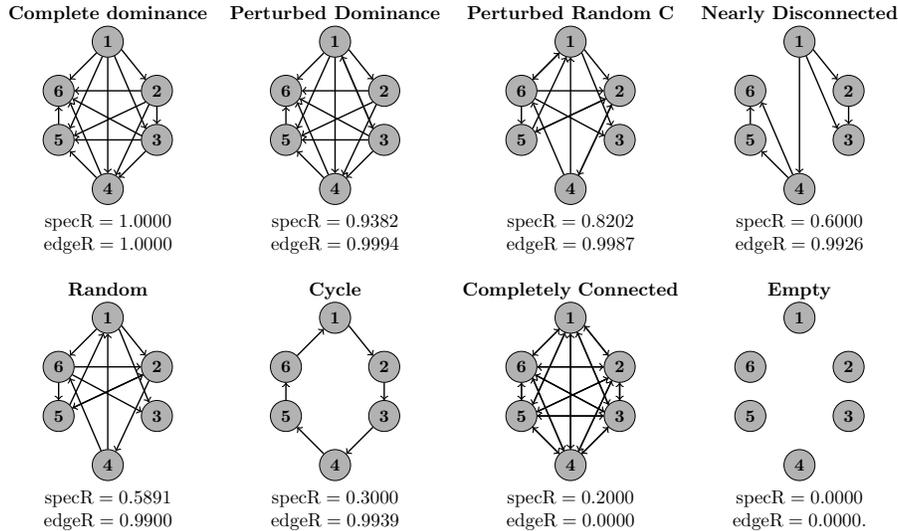
\begin{figure}[ht]
\centering
\resizebox{1.0\textwidth}{!}{
\begin{tabular}{cccc}
\textbf{Complete dominance} & \textbf{Perturbed Dominance} & \textbf{Perturbed Random C} & \textbf{Nearly Disconnected} \\
\resizebox{0.2\textwidth}{!}{
\begin{tikzpicture}
	\node[circle,draw=black,fill=black!30] (1) at (0,2) {\textbf{1}};
	\node[circle,draw=black,fill=black!30] (2) at (1,1) {\textbf{2}};
	\node[circle,draw=black,fill=black!30] (3) at (1,0) {\textbf{3}};
	\node[circle,draw=black,fill=black!30] (4) at (0,-1) {\textbf{4}};
	\node[circle,draw=black,fill=black!30] (5) at (-1,0) {\textbf{5}};
	\node[circle,draw=black,fill=black!30] (6) at (-1,1) {\textbf{6}};
	
	\draw[black,->,thick](1) to [out=330,in=135,looseness=0](2);
	\draw[black,->,thick](1) to [out=295,in=135,looseness=0](3);
	\draw[black,->,thick](1) to [out=270,in=90,looseness=0](4);
	\draw[black,->,thick](1) to [out=250,in=45,looseness=0](5);
	\draw[black,->,thick](1) to [out=225,in=45,looseness=0](6);
	\draw[black,->,thick](2) to [out=270,in=90,looseness=0](3);
	\draw[black,->,thick](2) to [out=240,in=60,looseness=0](4);
	\draw[black,->,thick](2) to [out=225,in=15,looseness=0](5);
	\draw[black,->,thick](2) to [out=180,in=0,looseness=0](6);
	\draw[black,->,thick](3) to [out=225,in=45,looseness=0](4);
	\draw[black,->,thick](3) to [out=180,in=0,looseness=0](5);
	\draw[black,->,thick](3) to [out=165,in=330,looseness=0](6);
	\draw[black,->,thick](4) to [out=135,in=315,looseness=0](5);
	\draw[black,->,thick](4) to [out=115,in=315,looseness=0](6);
	\draw[black,->,thick](5) to [out=90,in=270,looseness=0](6);
\end{tikzpicture}%
}\hfill
&
\resizebox{0.2\textwidth}{!}{
\begin{tikzpicture}
	\node[circle,draw=black,fill=black!30] (1) at (0,2) {\textbf{1}};
	\node[circle,draw=black,fill=black!30] (2) at (1,1) {\textbf{2}};
	\node[circle,draw=black,fill=black!30] (3) at (1,0) {\textbf{3}};
	\node[circle,draw=black,fill=black!30] (4) at (0,-1) {\textbf{4}};
	\node[circle,draw=black,fill=black!30] (5) at (-1,0) {\textbf{5}};
	\node[circle,draw=black,fill=black!30] (6) at (-1,1) {\textbf{6}};
	
	\draw[black,->,thick](1) to [out=330,in=135,looseness=0](2);
	\draw[black,->,thick](1) to [out=295,in=135,looseness=0](3);
	\draw[black,->,thick](1) to [out=270,in=90,looseness=0](4);
	\draw[black,->,thick](1) to [out=250,in=45,looseness=0](5);
	\draw[black,->,thick](1) to [out=225,in=45,looseness=0](6);
	\draw[black,->,thick](2) to [out=240,in=60,looseness=0](4);
	\draw[black,->,thick](2) to [out=225,in=15,looseness=0](5);
	\draw[black,->,thick](2) to [out=180,in=0,looseness=0](6);
	\draw[black,->,thick](3) to [out=135,in=295,looseness=0](1);
	\draw[black,->,thick](3) to [out=225,in=45,looseness=0](4);
	\draw[black,->,thick](3) to [out=180,in=0,looseness=0](5);
	\draw[black,->,thick](3) to [out=165,in=330,looseness=0](6);
	\draw[black,->,thick](4) to [out=135,in=315,looseness=0](5);
	\draw[black,->,thick](4) to [out=115,in=315,looseness=0](6);
	\draw[black,->,thick](5) to [out=90,in=270,looseness=0](6);
\end{tikzpicture}%
}\hfill
&
\resizebox{0.2\textwidth}{!}{
\begin{tikzpicture}
	\node[circle,draw=black,fill=black!30] (1) at (0,2) {\textbf{1}};
	\node[circle,draw=black,fill=black!30] (2) at (1,1) {\textbf{2}};
	\node[circle,draw=black,fill=black!30] (3) at (1,0) {\textbf{3}};
	\node[circle,draw=black,fill=black!30] (4) at (0,-1) {\textbf{4}};
	\node[circle,draw=black,fill=black!30] (5) at (-1,0) {\textbf{5}};
	\node[circle,draw=black,fill=black!30] (6) at (-1,1) {\textbf{6}};
	
	\draw[black,->,thick](1) to [out=330,in=135,looseness=0](2);
	\draw[black,->,thick](1) to [out=315,in=120,looseness=0](3);
	\draw[black,->,thick](1) to [out=225,in=45,looseness=0](6);
	\draw[black,->,thick](2) to [out=240,in=60,looseness=0](4);
	\draw[black,->,thick](2) to [out=210,in=30,looseness=0](5);
	\draw[black,->,thick](4) to [out=90,in=270,looseness=0](1);
	\draw[black,->,thick](4) to [out=60,in=240,looseness=0](2);
	\draw[black,->,thick](4) to [out=115,in=315,looseness=0](6);
	\draw[black,->,thick](5) to [out=45,in=255,looseness=0](1);
	\draw[black,->,thick](5) to [out=30,in=210,looseness=0](2);
	\draw[black,->,thick](6) to [out=45,in=225,looseness=0](1);
	\draw[black,->,thick](6) to [out=0,in=180,looseness=0](2);
	\draw[black,->,thick](6) to [out=330,in=165,looseness=0](3);
	\draw[black,->,thick](6) to [out=270,in=90,looseness=0](5);
\end{tikzpicture}%
}\hfill
&
\resizebox{0.2\textwidth}{!}{
\begin{tikzpicture}
	\node[circle,draw=black,fill=black!30] (1) at (0,2) {\textbf{1}};
	\node[circle,draw=black,fill=black!30] (2) at (1,1) {\textbf{2}};
	\node[circle,draw=black,fill=black!30] (3) at (1,0) {\textbf{3}};
	\node[circle,draw=black,fill=black!30] (4) at (0,-1) {\textbf{4}};
	\node[circle,draw=black,fill=black!30] (5) at (-1,0) {\textbf{5}};
	\node[circle,draw=black,fill=black!30] (6) at (-1,1) {\textbf{6}};
	
	\draw[black,->,thick](1) to [out=330,in=135,looseness=0](2);
	\draw[black,->,thick](1) to [out=295,in=135,looseness=0](3);
	\draw[black,->,thick](1) to [out=270,in=90,looseness=0](4);
	\draw[black,->,thick](2) to [out=270,in=90,looseness=0](3);
	\draw[black,->,thick](4) to [out=135,in=315,looseness=0](5);
	\draw[black,->,thick](4) to [out=115,in=315,looseness=0](6);
	\draw[black,->,thick](5) to [out=90,in=270,looseness=0](6);
\end{tikzpicture}%
}\\
$\specr = 1.0000$ & $\specr = 0.9382$ & $\specr = 0.8202$ & $\specr = 0.6000$ \\
$\edger = 1.0000$ & $\edger = 0.9994$ & $\edger = 0.9987$ & $\edger = 0.9926$ \\
& & & \\
\textbf{Random} & \textbf{Cycle} & \textbf{Completely Connected} & \textbf{Empty} \\
\resizebox{0.2\textwidth}{!}{
\begin{tikzpicture}
	\node[circle,draw=black,fill=black!30] (1) at (0,2) {\textbf{1}};
	\node[circle,draw=black,fill=black!30] (2) at (1,1) {\textbf{2}};
	\node[circle,draw=black,fill=black!30] (3) at (1,0) {\textbf{3}};
	\node[circle,draw=black,fill=black!30] (4) at (0,-1) {\textbf{4}};
	\node[circle,draw=black,fill=black!30] (5) at (-1,0) {\textbf{5}};
	\node[circle,draw=black,fill=black!30] (6) at (-1,1) {\textbf{6}};
	
	\draw[black,->,thick](1) to [out=330,in=135,looseness=0](2);
	\draw[black,->,thick](1) to [out=315,in=120,looseness=0](3);
	\draw[black,->,thick](1) to [out=225,in=45,looseness=0](6);
	\draw[black,->,thick](2) to [out=240,in=60,looseness=0](4);
	\draw[black,->,thick](2) to [out=210,in=30,looseness=0](5);
	\draw[black,->,thick](4) to [out=90,in=270,looseness=0](1);
	\draw[black,->,thick](4) to [out=115,in=315,looseness=0](6);
	\draw[black,->,thick](5) to [out=45,in=255,looseness=0](1);
	\draw[black,->,thick](5) to [out=30,in=210,looseness=0](2);
	\draw[black,->,thick](6) to [out=0,in=180,looseness=0](2);
	\draw[black,->,thick](6) to [out=330,in=165,looseness=0](3);
	\draw[black,->,thick](6) to [out=270,in=90,looseness=0](5);
\end{tikzpicture}%
}\hfill
&
\resizebox{0.2\textwidth}{!}{
\begin{tikzpicture}
	\node[circle,draw=black,fill=black!30] (1) at (0,2) {\textbf{1}};
	\node[circle,draw=black,fill=black!30] (2) at (1,1) {\textbf{2}};
	\node[circle,draw=black,fill=black!30] (3) at (1,0) {\textbf{3}};
	\node[circle,draw=black,fill=black!30] (4) at (0,-1) {\textbf{4}};
	\node[circle,draw=black,fill=black!30] (5) at (-1,0) {\textbf{5}};
	\node[circle,draw=black,fill=black!30] (6) at (-1,1) {\textbf{6}};
	
	\draw[black,->,thick](1) to [out=330,in=135,looseness=0](2);
	\draw[black,->,thick](2) to [out=270,in=90,looseness=0](3);
	\draw[black,->,thick](3) to [out=225,in=45,looseness=0](4);
	\draw[black,->,thick](4) to [out=135,in=315,looseness=0](5);
	\draw[black,->,thick](5) to [out=90,in=270,looseness=0](6);
	\draw[black,->,thick](6) to [out=45,in=225,looseness=0](1);
\end{tikzpicture}%
}\hfill
&
\resizebox{0.2\textwidth}{!}{
\begin{tikzpicture}
	\node[circle,draw=black,fill=black!30] (1) at (0,2) {\textbf{1}};
	\node[circle,draw=black,fill=black!30] (2) at (1,1) {\textbf{2}};
	\node[circle,draw=black,fill=black!30] (3) at (1,0) {\textbf{3}};
	\node[circle,draw=black,fill=black!30] (4) at (0,-1) {\textbf{4}};
	\node[circle,draw=black,fill=black!30] (5) at (-1,0) {\textbf{5}};
	\node[circle,draw=black,fill=black!30] (6) at (-1,1) {\textbf{6}};
	
	\draw[black,->,thick](1) to [out=330,in=135,looseness=0](2);
	\draw[black,->,thick](1) to [out=295,in=135,looseness=0](3);
	\draw[black,->,thick](1) to [out=270,in=90,looseness=0](4);
	\draw[black,->,thick](1) to [out=250,in=45,looseness=0](5);
	\draw[black,->,thick](1) to [out=225,in=45,looseness=0](6);
	\draw[black,->,thick](2) to [out=135,in=330,looseness=0](1);
	\draw[black,->,thick](2) to [out=270,in=90,looseness=0](3);
	\draw[black,->,thick](2) to [out=240,in=60,looseness=0](4);
	\draw[black,->,thick](2) to [out=225,in=15,looseness=0](5);
	\draw[black,->,thick](2) to [out=180,in=0,looseness=0](6);
	\draw[black,->,thick](3) to [out=135,in=295,looseness=0](1);
	\draw[black,->,thick](3) to [out=90,in=270,looseness=0](2);
	\draw[black,->,thick](3) to [out=225,in=45,looseness=0](4);
	\draw[black,->,thick](3) to [out=180,in=0,looseness=0](5);
	\draw[black,->,thick](3) to [out=165,in=330,looseness=0](6);
	\draw[black,->,thick](4) to [out=90,in=270,looseness=0](1);
	\draw[black,->,thick](4) to [out=60,in=240,looseness=0](2);
	\draw[black,->,thick](4) to [out=45,in=225,looseness=0](3);
	\draw[black,->,thick](4) to [out=135,in=315,looseness=0](5);
	\draw[black,->,thick](4) to [out=115,in=315,looseness=0](6);
	\draw[black,->,thick](5) to [out=45,in=250,looseness=0](1);
	\draw[black,->,thick](5) to [out=15,in=225,looseness=0](2);
	\draw[black,->,thick](5) to [out=0,in=180,looseness=0](3);
	\draw[black,->,thick](5) to [out=315,in=135,looseness=0](4);
	\draw[black,->,thick](5) to [out=90,in=270,looseness=0](6);
	\draw[black,->,thick](6) to [out=45,in=225,looseness=0](1);
	\draw[black,->,thick](6) to [out=0,in=180,looseness=0](2);
	\draw[black,->,thick](6) to [out=330,in=165,looseness=0](3);
	\draw[black,->,thick](6) to [out=315,in=115,looseness=0](4);
	\draw[black,->,thick](6) to [out=270,in=90,looseness=0](5);
\end{tikzpicture}%
}\hfill
&
\resizebox{0.2\textwidth}{!}{
\begin{tikzpicture}
	\node[circle,draw=black,fill=black!30] (1) at (0,2) {\textbf{1}};
	\node[circle,draw=black,fill=black!30] (2) at (1,1) {\textbf{2}};
	\node[circle,draw=black,fill=black!30] (3) at (1,0) {\textbf{3}};
	\node[circle,draw=black,fill=black!30] (4) at (0,-1) {\textbf{4}};
	\node[circle,draw=black,fill=black!30] (5) at (-1,0) {\textbf{5}};
	\node[circle,draw=black,fill=black!30] (6) at (-1,1) {\textbf{6}};
	
\end{tikzpicture}%
}\\
$\specr = 0.5891$ & $\specr = 0.3000$ & $\specr = 0.2000$ & $\specr = 0.0000$ \\
$\edger = 0.9900$ & $\edger = 0.9939$ & $\edger = 0.0000$ & $\edger = 0.0000$.
\end{tabular}%
}
\caption{Digraphs from~\cite{Anderson2019} with binary weights.}
\label{fig:simod-examples}
\end{figure}

Note that the Perturbed Random C digraph was used in~\cite{Anderson2019} to illustrate how certain edge changes can improve rankability.
In this particular example, the edges $(4,2)$ and $(6,1)$ were added to the Random digraph.
Moreover, for the considered digraphs, the rankability measures $\specr$ and $\edger$ have a strong correlation; in fact, the Spearman correlation coefficient between them is $0.92$.

\section{Rankability of Data Sets}\label{sec:rank-data}
In this section, we use $\specr$ to measure the rankability of several datasets from the world of chess and college football.
Also, we compare our rankability measure with the sensitivity and backward predictability of Elo ratings. 
All data sets and Python source code from this section are available online at~\url{https://github.com/trcameron/specR}. 

We note that the Elo rating system was developed by Aarpad Elo in the 1950's, it was adopted by the International Chess Federation in 1970, and since then has been adapted to rate other sports such as football, basketball, and soccer~\cite{Langville2012}.
For our purposes, each player (team) will have an Elo rating of $e = 0$ to start the tournament (season).
Then, after a match (game) between players (teams) $i$ and $j$, we update player (team) $i$ Elo rating as follows:
\begin{equation}\label{eq:elo-rating}
e_{i}^{(new)} = e_{i}^{(old)} + k(s-\mu),
\end{equation}
where $s=1$, $s=0$, or $s=1/2$ depending on if player (team) $i$ wins, looses, or ties, respectively. 
In addition, the parameter $k=40$ for chess and $k=32$ for college football.
Finally, we compute $\mu$ as follows:
\[
\mu=\frac{1}{1+10^{-d/\xi}},
\]
where $d=e_{i}^{(old)} - e_{j}^{(old)}$, and $\xi=400$ for chess and $\xi=1000$ for college football.

\subsection{Sinquefield Cup}\label{subsec:sinquefield-cup}
The Sinquefield Cup is an invite only round-robin chess tournament that is part of the Grand Chess Tour.
The first edition was held in 2013 as a double round-robin tournament comprised of 4 players; in 2014, it was a double round-robin tournament with 6 players.
Then, from 2015-2018, the Sinquefield cup was a single round-robin tournament between 10 players.
Most recently, in 2019, the tournament was a single round-robin between 12 players. 

For each year of the tournament, we have data that represents the outcome of the individual matches of each round.
We model this data as a digraph where the edge $(i,j)$ has weight $w_{ij}$ equal to the winning percentage of player $i$ over player $j$, where draws are counted as a half win of player $i$ over player $j$ and a half win of player $j$ over player $i$. 
For each round, we use $\specr$ to measure the rankability of this data; the results are shown in Figure~\ref{fig:SQFieldCup}.

\begin{figure}[ht]
\centering
\resizebox{0.50\textwidth}{!}{
\begin{tikzpicture}
	\begin{axis}[
		xlabel = Round Number,
		ylabel =  Rankability,
		legend pos = south east,
		cycle list name = black white]
		\addplot+[black] coordinates{
			(1,0.1667)
			(2,0.3333)
			(3,0.4339)
			(4,0.5305)
			(5,0.6555)
			(6,0.7917)
		};
		\addplot+[black!90] coordinates{
			(1,0.1000)
			(2,0.2000)
			(3,0.3000)
			(4,0.4000)
			(5,0.5000)
			(6,0.6000)
			(7,0.7000)
			(8,0.7583)
			(9,0.7832)
			(10,0.7500)
		};
		\addplot+[black!80] coordinates{
			(1,0.1111)
			(2,0.2222)
			(3,0.2873)
			(4,0.3686)
			(5,0.4093)
			(6,0.4578)
			(7,0.5761)
			(8,0.6208)
			(9,0.6711)
		};
		\addplot+[black!70] coordinates{
			(1,0.1111)
			(2,0.1944)
			(3,0.2611)
			(4,0.3325)
			(5,0.4220)
			(6,0.4846)
			(7,0.5152)
			(8,0.5625)
			(9,0.6067)
		};
		\addplot+[black!60] coordinates{
			(1,0.1111)
			(2,0.1944)
			(3,0.2639)
			(4,0.3666)
			(5,0.4328)
			(6,0.4889)
			(7,0.5395)
			(8,0.5719)
			(9,0.6593)
		};
		\addplot+[black!50] coordinates{
			(1,0.1111)
			(2,0.1944)
			(3,0.2662)
			(4,0.3319)
			(5,0.3876)
			(6,0.4882)
			(7,0.5391)
			(8,0.5857)
			(9,0.6389)
		};
		\addplot+[black!40] coordinates{
			(1,0.0909)
			(2,0.1591)
			(3,0.2220)
			(4,0.2826)
			(5,0.3336)
			(6,0.3847)
			(7,0.4252)
			(8,0.4668)
			(9,0.5219)
			(10,0.5572)
			(11,0.6021)
		};
		\legend{2013, 2014, 2015, 2016, 2017, 2018, 2019}
	\end{axis}
\end{tikzpicture}%
}
\caption{Round by Round Rankability of Sinquefield Cup}
\label{fig:SQFieldCup}
\end{figure}
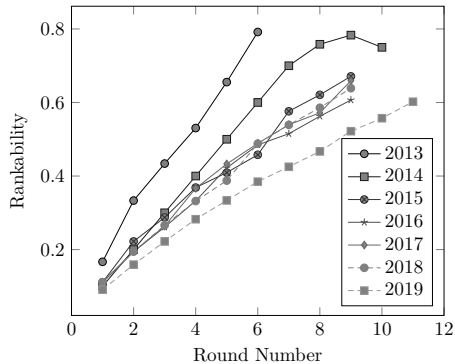

Note that as the number of rounds increases the rankability also increases, with the exception of the last round in 2014; this is expected since as the number of rounds increases so does the information about the players in the tournament.
In addition, note that the years 2013 and 2014 were significantly more rankable than the remaining years.
Moreover, both 2013 and 2014 have characteristics that are indicative of being more rankable. 
For instance, in 2013 there was a small pool of players and the final scorecard had uniformly distributed player totals. 
Furthermore, in 2014, Fabiano Caruana won his first seven matches, which was later noted as an ``historical achievement" by fellow competitor Levon Aronian.
Also, note that Caruana drew his last three matches, which helps explain the decrease in rankability at the end of 2014.

Conversely, the least rankable years are filled with point ties in the final scorecard and there is no clear ranking of all the players.
Specifically, in 2016, there was a four-way tie for second place and in 2018 there was a three-way tie for first place. 
Interestingly, the players tied for first decided to split the prize rather than take part in a tiebreaker due to other playing obligations.
Most recently, in 2019, there was a rapid blitz playoff for first place between Magnus Carlsen and Ding Liren who ended up winning.

In addition to measuring rankability, we computed the sensitivity of the Elo ratings.
In particular, given $n$ players in the tournament, we denote the $m$th round Elo ratings by
\[
\textbf{x}^{(m)} = \left[e_{1}^{(m)},\ldots,e_{n}^{(m)}\right],
\]
where $e_{i}^{(m)}$ is computed as described in~\eqref{eq:elo-rating}.
Then, for $m>1$, we denote by $y_{m}$ the correlation between $\textbf{x}^{(m)}$ and $\textbf{x}^{(m-1)}$, which is measured using the Spearman correlation coefficient.
Finally, we define the Elo rating correlation over the entire year by the weighted average
\[
2\frac{\sum_{m=1}^{r}(m-1)y_{m}}{r(r-1)},
\]
where $r$ is the number of rounds that year.
Essentially this weighted average places a larger emphasis on later rounds. 
We record the rankability at the end of the tournament and the yearly Elo rating correlation in Table~\ref{tab:SQFieldCup}.

\begin{table}
\centering
\resizebox{0.50\textwidth}{!}{
\begin{tabular}{ccc}
Year & Rankability & Rating Correlation \\
\hline
\rowcolor{black!30}2013 & 0.7917 & 0.9600 \\
2014 & 0.7500 & 0.9263 \\
2015 & 0.6711 & 0.8537 \\
2016 & 0.6067 & 0.8205 \\
2017 & 0.6593 & 0.8675 \\
2018 & 0.6389 & 0.9253 \\
\rowcolor{black!30}2019 & 0.6021 & 0.7875 \\
\end{tabular}
}
\caption{Elo Rating Correlation of SinquefieldCup}
\label{tab:SQFieldCup}
\end{table}

We have highlighted the rows with highest and lowest Elo rating correlation.
Note that these years also correspond to the highest and lowest rankability, respectively. 
There is a near monotonic relationship between the rankability and Elo rating correlation; the Spearman correlation coefficient between these two variables is $0.86$.

\subsection{Big East Football}\label{subsec:big-east-football}
We continue our testing on data from the Big East (1995-2012) football conference, which was a Division I collegiate football conference consisting of as many as 8 universities.
For each year, we model the games played week-by-week as a digraph with binary weights where the edge $(i,j)$ exists if team $i$ beat team $j$.
For each week, we update every team's Elo rating.
Then, at the end of the season we use $\specr$ to compute the rankability of the final digraph. 

As with the Sinquefield Cup, we measure the sensitivity of the Elo ratings.
In addition, we compute the backward predictability of the final Elo ratings, i.e., the percentage of games whose outcome is correctly determined by the Elo ratings and home-field advantage~\cite{Langville2012}.
The results are displayed in Table~\ref{tab:big-east}.

\begin{table}[ht]
\centering
\resizebox{0.70\textwidth}{!}{
\begin{tabular}{cccc}
Year & Rankability & Rating Correlation & Backward Predictability \\
\hline
1995	 & 0.8571 & 0.9253 & 0.8929 \\
1996	 & 0.8571 & 0.9477 & 0.9286 \\
1997	 & 0.8149 & 0.8487 & 0.75 \\
1998	 & 0.8169 & 0.8835 & 0.8571 \\
1999	 & 0.8571 & 0.8945 & 0.8571 \\
2000	 & 0.8571 & 0.9183 & 0.9286 \\
\rowcolor{black!30}2001	 & 0.8571 & 0.9397 & 0.9643 \\
2002	 & 0.8571 & 0.9525 & 0.9286 \\
2003	 & 0.8571 & 0.9075 & 0.8929 \\
2004	 & 0.6615 & 0.8077 & 0.7143 \\
2005	 & 0.8375 & 0.907 & 0.8571 \\
2006	 & 0.8049 & 0.9218 & 0.8214 \\
\rowcolor{black!30}2007	 & 0.6841 & 0.7985 & 0.7143 \\
2008	 & 0.8049 & 0.8803 & 0.8214 \\
2009	 & 0.8571 & 0.9392 & 0.8929 \\
2010	 & 0.7082 & 0.8434 & 0.75 \\
2011	 & 0.7143 & 0.7312 & 0.75 \\
2012	 & 0.7143 & 0.8874 & 0.7143 \\
\end{tabular}
}
\caption{Elo Rating Correlation and Backward Predictability of Big East Football}
\label{tab:big-east}
\end{table}

Note that the highest backward predictability corresponds with the highest rankability and, on the other hand, the lowest rankability corresponds with the lowest backward predictability.
There is a near monotonic relationship between the rankability and the Elo rating correlation and backward predictability, with a corresponding Spearman correlation coefficient of $0.89$ and $0.93$, respectively.

Figure~\ref{fig:big-east} displays the digraphs for years 2001 and 2007 of the Big East football conference.
Note that the nodes are labeled with respect to that team's Elo ranking.
Moreover, these digraphs clearly illustrate characteristics that are indicative of being more (2001) and less (2007) rankable.

In particular, note that in 2001 the best team won every game and the worst team lost every game.
In contrast, in 2007, the best team lost to the 6th best team, and the worst team beat the 6th best team.
Note that there is a cycle between the worst team, 6th best team, and best team. 
In fact, there are 205 total cycles in the digraph for 2007 as computed by Johnson's algorithm~\cite{Johnson1975}.
Finally, note that there is only one cycle in the 2001 digraph between the teams that rate 5th, 4th, and 3rd.

\begin{figure}[ht]
\centering
\begin{tabular}{ccc}
\resizebox{0.325\textwidth}{!}{
\begin{tikzpicture}
	\node[circle,draw=black,fill=black!30] (m) at (1,3) {\textbf{1}};
	\node[circle,draw=black,fill=black!30] (s) at (3,1) {\textbf{2}};
	\node[circle,draw=black,fill=black!30] (b) at (3,-1) {\textbf{3}};
	\node[circle,draw=black,fill=black!30] (v) at (1,-3) {\textbf{4}};
	\node[circle,draw=black,fill=black!30] (p) at (-1,-3) {\textbf{5}};
	\node[circle,draw=black,fill=black!30] (t) at (-3,-1) {\textbf{6}};
	\node[circle,draw=black,fill=black!30] (w) at (-3,1) {\textbf{7}};
	\node[circle,draw=black,fill=black!30] (r) at (-1,3) {\textbf{8}};
	
	\draw[black,->,thick](m) to [out=337.5,in=112.5,looseness=0](s);
	\draw[black,->,thick](m) to [out=292.5,in=112.5,looseness=0](b);
	\draw[black,->,thick](m) to [out=270,in=90,looseness=0](v);
	\draw[black,->,thick](m) to [out=247.5,in=67.5,looseness=0](p);
	\draw[black,->,thick](m) to [out=225,in=45,looseness=0](t);
	\draw[black,->,thick](m) to [out=202.5,in=22.5,looseness=0](w);
	\draw[black,->,thick](m) to [out=180,in=0,looseness=0](r);
	\draw[black,->,thick](s) to [out=270,in=90,looseness=0](b);
	\draw[black,->,thick](s) to [out=247.5,in=67.5,looseness=0](v);
	\draw[black,->,thick](s) to [out=225,in=45,looseness=0](p);
	\draw[black,->,thick](s) to [out=202.5,in=22.5,looseness=0](t);
	\draw[black,->,thick](s) to [out=180,in=0,looseness=0](w);
	\draw[black,->,thick](s) to [out=157.5,in=337.5,looseness=0](r);
	\draw[black,->,thick](b) to [out=202.5,in=22.5,looseness=0](p);
	\draw[black,->,thick](b) to [out=180,in=0,looseness=0](t);
	\draw[black,->,thick](b) to [out=157.5,in=337.5,looseness=0](w);
	\draw[black,->,thick](b) to [out=135,in=315,looseness=0](r);
	\draw[black,->,thick](v) to [out=45,in=225,looseness=0](b);
	\draw[black,->,thick](v) to [out=157.5,in=337.5,looseness=0](t);
	\draw[black,->,thick](v) to [out=135,in=315,looseness=0](w);
	\draw[black,->,thick](v) to [out=112.5,in=292.5,looseness=0](r);
	\draw[black,->,thick](p) to [out=0,in=180,looseness=0](v);
	\draw[black,->,thick](p) to [out=135,in=315,looseness=0](t);
	\draw[black,->,thick](p) to [out=112.5,in=292.5,looseness=0](w);
	\draw[black,->,thick](p) to [out=90,in=270,looseness=0](r);
	\draw[black,->,thick](t) to [out=90,in=270,looseness=0](w);
	\draw[black,->,thick](t) to [out=67.5,in=247.5,looseness=0](r);
	\draw[black,->,thick](w) to [out=45,in=215,looseness=0](r);
\end{tikzpicture}%
}
&
\hfill
&
\resizebox{0.325\textwidth}{!}{
\begin{tikzpicture}
	\node[circle,draw=black,fill=black!30] (w) at (1,3) {\textbf{1}};
	\node[circle,draw=black,fill=black!30] (co) at (3,1) {\textbf{2}};
	\node[circle,draw=black,fill=black!30] (ci) at (3,-1) {\textbf{3}};
	\node[circle,draw=black,fill=black!30] (sf) at (1,-3) {\textbf{4}};
	\node[circle,draw=black,fill=black!30] (r) at (-1,-3) {\textbf{5}};
	\node[circle,draw=black,fill=black!30] (l) at (-3,-1) {\textbf{6}};
	\node[circle,draw=black,fill=black!30] (p) at (-3,1) {\textbf{7}};
	\node[circle,draw=black,fill=black!30] (s) at (-1,3) {\textbf{8}};

	\draw[black,->,thick](w) to [out=337.5,in=112.5,looseness=0](co);
	\draw[black,->,thick](w) to [out=292.5,in=112.5,looseness=0](ci);
	\draw[black,->,thick](w) to [out=247.5,in=67.5,looseness=0](r);
	\draw[black,->,thick](w) to [out=225,in=45,looseness=0](l);
	\draw[black,->,thick](w) to [out=180,in=0,looseness=0](s);
	\draw[black,->,thick](co) to [out=247.5,in=67.5,looseness=0](sf);
	\draw[black,->,thick](co) to [out=225,in=45,looseness=0](r);
	\draw[black,->,thick](co) to [out=202.5,in=22.5,looseness=0](l);
	\draw[black,->,thick](co) to [out=180,in=0,looseness=0](p);
	\draw[black,->,thick](co) to [out=157.5,in=337.5,looseness=0](s);
	\draw[black,->,thick](ci) to [out=90,in=270,looseness=0](co);
	\draw[black,->,thick](ci) to [out=225,in=45,looseness=0](sf);
	\draw[black,->,thick](ci) to [out=202.5,in=22.5,looseness=0](r);
	\draw[black,->,thick](ci) to [out=135,in=315,looseness=0](s);
	\draw[black,->,thick](sf) to [out=90,in=270,looseness=0](w);
	\draw[black,->,thick](sf) to [out=157.5,in=337.5,looseness=0](l);
	\draw[black,->,thick](sf) to [out=135,in=315,looseness=0](p);
	\draw[black,->,thick](sf) to [out=112.5,in=292.5,looseness=0](s);
	\draw[black,->,thick](r) to [out=0,in=180,looseness=0](sf);
	\draw[black,->,thick](r) to [out=112.5,in=292.5,looseness=0](p);
	\draw[black,->,thick](r) to [out=90,in=270,looseness=0](s);
	\draw[black,->,thick](l) to [out=0,in=180,looseness=0](ci);
	\draw[black,->,thick](l) to [out=315,in=135,looseness=0](r);
	\draw[black,->,thick](l) to [out=90,in=270,looseness=0](p);
	\draw[black,->,thick](p) to [out=22.5,in=202.5,looseness=0](w);
	\draw[black,->,thick](p) to [out=337.5,in=157.5,looseness=0](ci);
	\draw[black,->,thick](p) to [out=45,in=215,looseness=0](s);
	\draw[black,->,thick](s) to [out=247.5,in=67.5,looseness=0](l);
\end{tikzpicture}%
}
\end{tabular}
\caption{Big East Football 2001 (left) and 2007 (right) Digraphs}
\label{fig:big-east}
\end{figure}
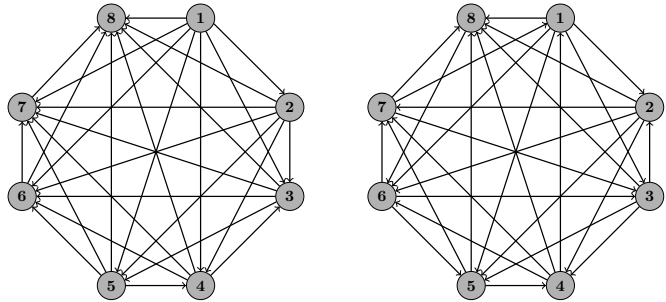

\section{Conclusion}\label{sec:conclusion}
The spectral characterization of acyclic digraphs and the spectral-degree characterization of complete dominance graphs in Theorem~\ref{thm:acyclic-spec} and Corollary~\ref{cor:dom-eigval}, respectively, lead to the rankability measure $\specr$.
This measure is cost-effective and more widely applicable than the rankability measure in~\cite{Anderson2019}.
Our measure is outlined in Algorithm~\ref{alg:specr}, and we support the details of our algorithm with several results regarding the Laplacian spectrum of complete dominance graphs.

In particular, Corollary~\ref{cor:dom-eigval-cond} implies that the Laplacian spectrum of complete dominance graphs is well-conditioned.
Furthermore, by Theorem~\ref{thm:haus-pert-edge}, a single edge change of a complete dominance graph results in a Hausdorff distance equal to the weight of that edge.
Finally, by Theorem~\ref{thm:haus-upper-bound}, the Hausdorff distance between the Laplacian spectrum of a complete dominance graph and any other simple digraph on $n$ vertices with weights between zero and one is bound above by $(n-1)$.

In Section~\ref{sec:rank-data}, we use $\specr$ to analyze the rankability of datasets from the world of chess and college football. 
Moreover, we demonstrate that $\specr$ has a strong correlation with the sensitivity and backward predictability of Elo ratings.

Since $\specr$ is founded upon a comparison to complete dominance graphs, this measure struggles to give meaningful results for sparse data, i.e., data where not all possible comparisons are explored.
Future research includes generalizing our measure, or the development of new measures, to allow for sparse data. 
In particular, we are interested in investigating other graph properties that can be used to help measure rankability, e.g., the algebraic connectivity as defined for digraphs in~\cite{Wu2005-1}.

\section*{Acknowledgements}
The authors wish to acknowledge Paul Anderson, Kathryn Behling, and Tim Chartier for many stimulating conversations that helped motivate some of the topics in this article.
We are particularly grateful to Tim Chartier for providing the college football data used in this article.
In addition, we wish to acknowledge an anonymous referee whose thoughtful comments greatly improved this article.


\end{document}